\documentclass[a4paper, 12pt]{amsart}

\usepackage{amssymb,amsmath}
\usepackage[dvips]{graphics}
\usepackage[dvips]{color}
\usepackage{graphicx}
\usepackage[english]{babel}
\usepackage[latin1]{inputenc}
\usepackage{comment}
\usepackage{todonotes}

\newtheorem{theorem}{Theorem}[section]
\newtheorem{lemma}[theorem]{Lemma}

\newtheorem{corollary}[theorem]{Corollary}

\theoremstyle{definition}
\newtheorem{definition}[theorem]{Definition}

\newtheorem{remark}[theorem]{Remark}

\numberwithin{equation}{section}

\title[Generalized Lebesgue points for Haj\l asz functions]
{
Generalized Lebesgue points for  Haj\l asz 
functions 
}
\author{Toni Heikkinen}

\DeclareMathOperator*{\essinf}{ess\,inf}
\DeclareMathOperator*{\esssup}{ess\,sup}

\newcommand\rn{\mathbb R^n}
\newcommand\re{\mathbb R}
\newcommand\rv{\overline{\mathbb R}}
\newcommand\n{\mathbb N}

\newcommand\ph{\varphi}
\newcommand\eps{\varepsilon}
\newcommand\M{\operatorname{\mathcal M}}
\newcommand\cA{\mathcal A}

\newcommand\cD{\mathcal D}

\newcommand\cP{\mathcal P}

\providecommand{\ch}[1]{\text{\raise 2pt \hbox{$\chi$}\kern-0.2pt}_{#1}}

\providecommand{\vint}[1]{\mathchoice
          {\mathop{\vrule width 5pt height 3 pt depth -2.5pt
                  \kern -9.5pt \kern 1pt\intop}\nolimits_{\kern -5pt{#1}}}%
          {\mathop{\vrule width 5pt height 3 pt depth -2.6pt
                  \kern -6pt \intop}\nolimits_{\kern -3pt{#1}}}%
          {\mathop{\vrule width 5pt height 3 pt depth -2.6pt
                  \kern -6pt \intop}\nolimits_{\kern -3pt{#1}}}%
          {\mathop{\vrule width 5pt height 3 pt depth -2.6pt
                  \kern -6pt \intop}\nolimits_{\kern -3pt{#1}}}}
\begin{document}

\begin{abstract}
Let $X$ be a quasi-Banach function space over a doubling metric measure space $\mathcal P$. Denote by $\alpha_X$ the generalized upper Boyd index of $X$. We show that if $\alpha_X<\infty$
and $X$ has absolutely continuous quasinorm, then quasievery point is a generalized Lebesgue point of a quasicontinuous Haj\l asz function $u\in\dot M^{s,X}$.
Moreover,
if $\alpha_X<(Q+s)/Q$, then quasievery point is a Lebesgue point of $u$. As an application we obtain Lebesgue type theorems for Lorentz--Haj\l asz, Orlicz--Haj\l asz and variable exponent Haj\l asz functions.
\end{abstract}

\keywords{Haj\l asz space, metric measure space, median, quasicontinuity, Lebesgue point, quasi-Banach function space}
\subjclass[2010]{46E35, 43A85} 

\date{\today}

\maketitle

\section{Introduction and main results}
Let $\cP=(\cP,d,\mu)$ be a doubling metric measure space. By the Lebesgue differentiation theorem, almost every point of a locally integrable function is a Lebesgue point. 
As expected, for smoother functions, the set of non-Lebesgue points is smaller. 
In \cite{KL}, Kinnunen and Latvala showed that for a quasicontinuous Haj\l asz--Sobolev function $u\in M^{1,p}(\cP)$, $p>1$, there exists a set $E$ of $M^{1,p}$-capacity zero such that
\begin{equation}\label{Leb point}
 \lim _{r\to0}\frac1{\mu(B(x,r))}\int_{B(x,r)}u(y)\,d\mu(y)=u(x)   
\end{equation}
for every $x\in\cP\setminus E$.
The case $p=1$ was studied in \cite{KiTu} and \cite{P}. Recently, in \cite{HeKoTu} and independently in \cite{BoKr}, it was shown that
if
\begin{equation}\label{Q}
\frac{\mu(B(y,r))}{\mu(B(x,R))}\ge c\Big(\frac rR\Big)^Q 
\end{equation}
for every $x\in\cP$, $0<r\le R$ and $y\in  B(x,R)$, then for every quasicontinuous $u\in\dot M^{s,p}$ with
$p>Q/(Q+s)$,  \eqref{Leb point} holds true outside a set of $M^{s,p}$-capacity zero.

If we replace integral averages in \eqref{Leb point} by medians, then the result holds true also for small $p>0$. For $0<\gamma<1$, $A\subset\cP$ and $u\in L^0$, denote 
\[
m_u^\gamma(A)=\inf\big\{a\in\re: \mu(\{x\in A: u(x)>a\})< \gamma\mu(A)\big\}.
\]
If $p>0$ and $u\in \dot M^{s,p}$ is quasicontinuous, then
by \cite[Theorem 1.2]{HeKoTu},
there exists a set $E\subset \cP$  of  $M^{s,p}$-capacity zero
such that 
\begin{equation}\label{gen leb point}
\lim_{r\to 0}m_u^{\gamma}(B(x,r))=u(x),
\end{equation}
for every $x\in \cP\setminus E$ and $0<\gamma\le 1/2$.

In this paper, we will study the existence of (generalized) Lebesgue points for functions $u$ whose Haj\l asz gradient belongs to a general quasi-Banach function space $X$.
This approach allows us to simultaneously cover, for example, Orlicz--Haj\l asz, Lorentz--Haj\l asz and variable exponent Haj\l asz functions.

For $0<\gamma<1$, $0<R\le\infty$,  $u\in L^0(\cP)$ and $x\in\cP$, denote
\[
\M^\gamma_R u(x)=\sup_{0<r<R}m^\gamma_{|u|}(B(x,r)).
\]
Operator $\M^\gamma=\M^\gamma_\infty$ and its variants have turned out to be useful in harmonic analysis and in the theory of function spaces, see for example \cite{FZ}, \cite{Fu}, \cite{GKZ}, \cite{HIT}, \cite{HeKi}, \cite{HeKoTu}, \cite{HeTu2}, \cite{Hy}, \cite{Ha}, \cite{JPW}, \cite{JT}, \cite{J}, \cite{Kar}, \cite{L},  \cite{LP}, \cite{LP2}, \cite{PT}, \cite{St}, \cite{StTo}, \cite{Zh}. 

\begin{theorem}\label{main thm}
Let $\mu$ be doubling. Suppose that $X$ has absolutely continuous quasinorm and that, for every $0<\gamma<1$ and for every ball $B\subset\cP$, there exists a constant $C$ such that
\begin{equation}\label{M gamma bd}
    \|(\M^\gamma_1 g)\chi_B\|_X\le C\|g\|_X
\end{equation}
for every $g\in X$. Let $0<s\le 1$.
Then, for every quasicontinuous $u\in\dot M^{s,X}(\cP)$,
there exists a set $E\subset \cP$  with  $C_{M^{s,X}}(E)=0$
such that 
\begin{equation}\label{claim}
   \lim_{r\to 0} m^\gamma_{|u-u(x)|}(B(x,r))=0\ \  \text{ and }\ \ 
   \lim_{r\to 0} m^\gamma_{u}(B(x,r))=u(x)
\end{equation}
for every $x\in \cP\setminus E$ and $0<\gamma< 1$.
\end{theorem}

We say that a point $x\in\cP$ satisfying \eqref{claim} for every $0<\gamma<1$ is a \emph{generalized Lebesgue point} of $u$.

\medskip
The (restricted) Hardy--Littlewood maximal function of a locally integrable function $u$ is
\[
\M_R u(x)=\sup_{0<r<R}\frac1{\mu(B(x,r))}\int_{B(x,r)}|u(y)|\,d\mu(y).
\]
As usual, we denote $\M=\M_\infty$.

\begin{theorem}\label{main thm 2} Let $0<s\le 1$ and let $\mu$ satisfy \eqref{Q}. Denote $q=Q/(Q+s)$.
Suppose that $X$ has absolutely continuous quasinorm and that, for every ball $B\subset \cP$, there exists a constant $C$ such that
\begin{equation}\label{M q-bd}
\|(\M_1g^q)^{1/q}\chi_B\|_X\le C\|g\|_X
\end{equation}
for every $0\le g\in X$.
Then, for every quasicontinuous $u\in\dot M^{s,X}(\cP)$, quasievery point is a Lebesgue point of $u$.
\end{theorem}

For any quasi-Banach function space $X$ over $\cP$, define $\Phi_X:(0,1)\to [1,\infty];$
\[
\Phi_X(\gamma)=\sup_{\|f\|_X\le 1}\|\M^\gamma f\|_X.
\]
The \emph{generalized upper Boyd index} of $X$ is
\begin{equation}
\alpha_X=\lim_{\gamma\to 0}\frac{\log \Phi_X(\gamma)}{\log(1/\gamma)}.
\end{equation}
 The generalized upper Boyd index was introduced in \cite{LP2}, where it was shown that the Hardy--Littlewood maximal operator is bounded on $X(\rn)$ if and only if
$\alpha_X<1$. 
As a corollary of Theorems \ref{main thm} ja \ref{main thm 2} we have the following result.


\begin{theorem}\label{main thm 3}
Suppose that $X$ has absolutely continuous quasinorm and that $\mu$ satisfies \eqref{Q}.
Let $0<s\le 1$ and let $u\in\dot M^{s,X}(\cP)$ be quasicontinuous.
\begin{itemize}
    \item[(1)] If $\alpha_X<\infty$, then quasievery point is a generalized Lebesgue point of $u$.
    \item[(2)]  If $\alpha_X<(Q+s)/Q$, then quasievery point is a Lebesgue point of $u$.
    \end{itemize}

\end{theorem}

\section{Preliminaries}\label{sec: preliminaries}

\subsection{Basic assumptions}
In this paper,  $\cP=(\cP, d,\mu)$ is a metric measure space equipped with a metric $d$ and a Borel regular,
doubling outer measure $\mu$, for which the measure of every ball is positive and finite.
The doubling property means that there exists a constant $c_d>0$ such that
\[
\mu(B(x,2r))\le c_d\mu(B(x,r))
\]
for every ball $B(x,r)=\{y\in \cP:d(y,x)<r\}$, where $x\in \cP$ and $r>0$.

The doubling condition is equivalent to the existence of  constants $c$ and $Q$ such that \eqref{Q} holds
for every $0<r\le R$, $x\in\cP$ and $y\in B(x,R)$. 

{\em The integral average} of a locally integrable function $u$ over a set $A$ of positive and finite measure is
\[
u_A=\vint{A}u\,d\mu=\frac{1}{\mu(A)}\int_A u\,d\mu.
\]

By $\ch{E}$, we denote the characteristic function of a set $E\subset \cP$ and by $\rv$, the extended real numbers $[-\infty,\infty]$.
$L^0=L^0(\cP)$ is the set of all measurable, almost everywhere finite functions $u\colon \cP\to\rv$.
In general, $C$ and $c$ are positive constants whose values are not necessarily same at each occurrence.
When we want to stress that the constant depends on other constants or parameters $a,b,\dots$, we write $C=C(a,b,\dots)$.
\subsection{Quasi-Banach function spaces} A quasinorm on a subspace of $L^0(\cP)$ is a functional $\|\cdot\|$ such that
\begin{itemize}
\item $\|f\|=0 \iff f=0$ a.e.; 
\item $\|\alpha f\|=|\alpha|\|f\|$ for every $\alpha\in\re$;
\item There exists a constant $c_\Delta$ such that $\|f+g\|\le c_\Delta(\|f\|+\|g\|)$.
\end{itemize}
A quasi-Banach function space $X$ over $\cP$ is a subspace of $L^0(\cP)$ equipped with
a complete quasinorm $\|\cdot\|_X$ that has the following properties:
\begin{itemize}
\item $g\in X$ and $|f|\le |g|$ a.e.$\implies $ $f\in X$ and $\|f\|_X\le \|g\|_X$;
\item $\mu(E)<\infty\implies \|\chi_E\|_X<\infty$;
\item $0\le f_k\uparrow f$ a.e.$\implies \|f_k\|_X\uparrow \|f\|_X$.
\end{itemize}
By the Aoki--Rolewicz theorem (\cite{A}, \cite{R}), there exists a constant $0<\rho\le 1$ such that
\begin{equation}\label{Aoki}
\|\sum_{k=1}^\infty f_k\|_X\le 4^{1/\rho}(\sum_{k=1}^\infty\|f_k\|_X^\rho)^{1/\rho}
\end{equation}
for all $f_1,f_2,\dots\in X$.

A quasinorm $\|\cdot\|_X$ on $X$ is absolutely continuous if $\|f\chi_{E_k}\|_X\to 0$ whenever $f\in X$ and $(E_k)_{k=1}^\infty$ is a decreasing sequence of sets such that $\cap_{k=1}^\infty E_k=\emptyset$.

\subsection{Haj\l asz  spaces}
 
Let $0<s<\infty$. 
A measurable function $g:\cP\to [0,\infty]$ is an  {\em $s$-gradient} of a function $u\in L^0(\cP)$ if there exists a set
$E\subset \cP$ with $\mu(E)=0$ such that for all $x,y\in X\setminus E$,
\begin{equation}\label{eq: gradient}
|u(x)-u(y)|\le d(x,y)^s(g(x)+g(y)).
\end{equation}
The collection of all $s$-gradients of $u$ is denoted by $\cD^s(u)$.

The {\em homogeneous Haj\l asz space} $\dot{M}^{s,X}=\dot{M}^{s,X}(\cP)$ consists of measurable functions $u$ for which
\[
\|u\|_{\dot M^{s,X}}=\inf_{g\in\mathcal{D}^s(u)}\|g\|_X
\]
is finite.
The {\em Haj\l asz space} $M^{s,X}=M^{s,X}(\cP)$ is $\dot M^{s,X}\cap X$  equipped with the norm
\[
\|u\|_{M^{s,X}}=\|u\|_{X}+\|u\|_{\dot M^{s,X}}.
\]

Haj\l asz spaces $M^{s,p}(\cP)=M^{s,L^p}(\cP)$, were intruduced in \cite{H} for $s=1$, $p\ge 1$ and in \cite{Y} for fractional scales. 
Recall that for $p>1$, $M^{1,p}(\rn)=W^{1,p}(\rn)$, see \cite{H}, whereas for $n/(n+1)<p\le 1$, $M^{1,p}(\rn)$ coincides with the Hardy--Sobolev space $H^{1,p}(\rn)$ by \cite[Thm 1]{KS}.


Next two lemmas for $s$-gradients follow easily from the definition, see
\cite[Lemmas 2.4]{KiMa} and \cite[Lemma 2.6]{KL}. 

\begin{lemma}\label{max lemma}
Let $u,v\in L^0(\cP)$, $g\in\cD^s(u)$ and $h\in\cD^s(v)$. 
Then $\max\{g,h\}$ is an $s$-gradient of functions $\max\{u,v\}$ and $\min\{u,v\}$.
\end{lemma}


\begin{lemma}\label{sup lemma} 
Let $u_i\in L^0(\cP)$ and $g_i\in \cD^s(u_i)$, $i\in\n$. 
Let $u=\sup_{i\in \n} u_i$ and $g=\sup_{i\in\n}g_{i}$.
If $u\in L^0(\cP)$, then $g\in \cD^s(u)$.
\end{lemma}

The following lemma is essentially \cite[Lemma 7.2]{HeKoTu}.

\begin{lemma}\label{gradient for product} 
Let $0<s\le 1$ and let $S\subset \cP$ be a measurable set.
Let $u\colon \cP\to\re$ be a measurable function with $g\in \cD^s(u)$ and
let $\ph$ be a bounded $L$-Lipschitz function supported in $S$.
Then 
\[
h=\big(\|\varphi\|_{\infty} g+(2\|\varphi\|_{\infty}\big)^{1-s}L^s |u|)\ch S \in \cD^s(u\varphi).
\]
Consequently, there exists a constant $C=C(s,\|\varphi\|_\infty,L)$ such that
\[
\|u\ph\|_{M^{s,X}}\le C(\|u\chi_S\|_X+\inf_{g\in \cD^s(u)}\|g\chi_S\|_X)
\]
for every $u\in \dot M^{s,X}$.
\end{lemma}

\begin{lemma}\label{cont lemma} Let $0<s\le 1$ and suppose that $X$ has absolutely continuous quasi-norm. Then s-H\"older continuous functions are dense in $M^{s,X}$.
\end{lemma}
\begin{proof}
Let $u\in M^{s,X}, g\in \cD^s(u)\cap X$ and let $E$ be the exceptional set for \eqref{eq: gradient}.
Then $u$ is $s$-H\"older continuous with constant $2\lambda$ in the set $E_\lambda=\{x\in \cP\setminus E: g(x)\le\lambda\}$. By \cite{McShane}, there is an extension $u_\lambda$ of $u_{|E_\lambda}$
to $\cP$ such that $u_\lambda$ is $s$-H\"older continuous with constant $2\lambda$.
It is easy to see that $(g+3\lambda)\chi_{\cP\setminus E_\lambda}\in\cD^s(u-u_\lambda)$, see \cite[Proposition 4.5]{SYY}. By the absolute continuity of $\|\cdot\|_X$, $\|(g+3\lambda)\chi_{\cP\setminus E_\lambda}\|_X\to 0$ and $\|u-u_\lambda\|_X\le 2\|u\chi_{\cP\setminus E_\lambda}\|_X\to 0$
as $\lambda\to 0$.
\end{proof}






\subsection{$\gamma$-median} 

Let $0<\gamma<1$. The $\gamma$-median of a measurable function
$u\colon \cP\to\rv$ over a set $A$ of finite measure is
\[
m_u^\gamma(A)=\inf\big\{a\in\re: \mu(\{x\in A: u(x)>a\})< \gamma\mu(A)\big\}.
\]
Note that if $u\in L^0(A)$ and $0<\mu(A)<\infty$, then $m_u^\gamma(A)$ is finite.

In the following lemma, we list some basic properties of the $\gamma$-median. 
Properties (a), (b), (d), (f) and (g) follow from \cite[Propositions 1.1 and 1.2]{PT} and (h) and (i) from \cite[Theorem 2.1]{PT}. The remaining properties (c) and (e) follow immediately from the definition.

\begin{lemma}\label{median lemma} 
The $\gamma$-median has the following properties:
\begin{itemize}
\item[(a)] If $\gamma\le\gamma'$, then $m_{u}^{\gamma}(A)\ge m_{u}^{\gamma'}(A)$.
\item[(b)] If $u\le v$ almost everywhere, then  $m_{u}^{\gamma}(A)\le m_{v}^{\gamma}(A)$.
\item[(c)] If $A\subset B$ and $\mu(B)\le C\mu(A)$, then $m_{u}^{\gamma}(A)\le m_{u}^{\gamma/C}(B)$.
\item[(d)] If $c\in\mathbb{R}$, then $m_u^\gamma(A)+c=m_{u+c}^\gamma(A)$.
\item[(e)] If $c>0$, then $m_{c\,u}^\gamma(A)=c\,m_{u}^\gamma(A)$.
\item[(f)] $|m_{u}^\gamma(A)|\le m_{|u|}^{\min\{\gamma,1-\gamma\}}(A)$. 
\item[(g)] $m_{u+v}^\gamma(A)\le m_{u}^{\gamma/2}(A)+m_{v}^{\gamma/2}(A)$.
\item[(h)] If $u$ is continuous, then for every $x\in \cP$,
\[
\lim_{r\to 0} m_{u}^\gamma(B(x,r))=u(x).
\]
\item[(i)]  If $u\in L^0(\cP)$, then 
there exists a set $E$ with $\mu(E)=0$ such that \[\lim_{r\to 0} m_{u}^\gamma(B(x,r))=u(x)\]
 for every $0<\gamma<1$ and $x\in \cP\setminus E$. 
\end{itemize}
\end{lemma}



\subsection{Discrete maximal functions}
In this subsection, we define ''discrete'' versions of the Hardy--Littlewood maximal function
\[
\M_R u(x)=\sup_{0<r<R}\, |u|_{B(x,r)}.
\]
and the median maximal function
\[
\M^\gamma_R u(x)=\sup_{0<r<R}\,m_{|u|}^\gamma(B(x,r)).
\]

We first recall the definition of a discrete convolution. 
Discrete convolutions are basic tools in harmonic analysis in homogeneous spaces, see for example \cite{CW} and \cite{MS}.
The following lemma is well known.

\begin{lemma}\label{partition of unity}
For every $r>0$, there exists a collection of balls $\{B_i=B(x_i,r):i\in I\}$, where $ 
I\subset\n$, and functions $\ph_i:\cP\to [0,1]$, $i\in I$, with the following properties: 
\begin{itemize}
\item[(a)] The balls $B(x_i,r/2)$, $i\in I$, cover $\cP$.
\item[(b)] $\sum_{i\in I}\ch{2B_i}\le C$.
\item[(c)] For every $i\in I$, $\ph_i$ is $C/r$-Lipschitz, $\ph_i\ge C^{-1}$ on $B_i$ and $\ph_i=0$ outside $2B_i$.
\item[(d)] $\sum_{i\in I}\ph_i=1$. 
\end{itemize}
Here, the constant $C$ depends only on the doubling constant $c_d$.

\end{lemma}
For each scale $r>0$, we choose a collection of balls $\{B_i:i\in I\}$ and a collection of functions $\{\ph_i:i\in I\}$ satisfying conditions (a)--(d) of Lemma \ref{partition of unity}.

\begin{definition}\label{def M*}
The \emph{discrete convolution} of a locally integrable function $u$ at the scale $r$  is
\begin{equation}\label{discrete convolution}
u_r=\sum_{i\in I}u_{B_i}\ph_i,
\end{equation}
where $\{B_i:i\in I\}$ and $\{\ph_i:i\in I\}$ are the chosen collections of balls and functions for the scale $r$.

The \emph{discrete  maximal function} of $u$ is $\M^*_R u\colon \cP\to\rv$,
\[
\M^*_R u(x)=\sup_{q\in \mathbb{Q}, 0<q<R}|u|_q(x).
\]
\end{definition}
The discrete maximal function, which can be seen as a smooth version of the Hardy--Littlewood maximal function, was introduced in \cite{KL}. 

Similarly, we define median versions of a discrete convolution and a discrete maximal function. 

\begin{definition}\label{def M gamma *}
Let $0<\gamma<1$. 
The \emph{discrete $\gamma$-median convolution} of a function $u\in L^0$ at scale $r>0$ is
\[
u_r^\gamma=\sum_{i\in I}m_{u}^\gamma(B_i)\ph_i,
\]
where $\{B_i:i\in I\}$ and $\{\ph_i:i\in I\}$ are as in the definition
\ref{def M*}.

The \emph{discrete $\gamma$-median maximal function} of $u$ is $\M^{\gamma,*}_R u\colon \cP\to\rv$,
\[
\M^{\gamma,*}_R u(x)=\sup_{q\in \mathbb{Q}, 0<q<R}|u|_{q}^\gamma(x).
\]
\end{definition}
As a supremum of continuous functions, the discrete maximal functions are lower semicontinuous and hence measurable.


\begin{lemma}\label{comparability}
Let $0<R\le\infty$. There exists a constant $C=C(c_d)\ge 1$ such that
\begin{equation}\label{comparability for maximal functions}
C^{-1}\M_{R/2} u\le \M^*_R u\le C\M_{3R} u
\end{equation}
for every $u\in L^1_{\text{loc}}$ and
\begin{equation}\label{comparability for median maximal functions}
\M^{\gamma}_{R/2} u\le C\M^{\gamma/C,*}_{R}u \ \ \text{ and }\ \ 
\M^{\gamma,*}_{R}u \le C\M^{\gamma/C}_{3R}u
\end{equation}
for every $u\in L^0$ and $0<\gamma<1$.
\end{lemma}
\begin{proof}
We will prove \eqref{comparability for median maximal functions}. The proof of \eqref{comparability for maximal functions} is similar.
Let $x\in\cP$, $r>0$ and let $u^\gamma_r=\sum_{i\in I}m_{u}^\gamma(B_i)\ph_i$ be as in the definition \ref{def M gamma *}.
If $x\in 2B_i$, then $B_i\subset B(x,3r)$. By the doubling property,
$\mu(B(x,3r))\le C\mu(B_i)$, and so, by Lemma \ref{median lemma}(c),
$m^\gamma_u(B_i)\le m^{\gamma/C}_u(B(x,3r))$. Since $\sum_{i\in I}\chi_{2B_i}(x)\le C$, it follows that
$$u^\gamma_r(x)\le C m^{\gamma/C}(B(x,3r)).$$

On the other hand, since the balls $\frac12 B_i$, $i\in I$, cover $\cP$, there is $i\in I$ such that $B(x,r/2)\subset B_i$. By the doubling property,
$\mu(B_i)\le C\mu(B(x,r/2))$, and so, by Lemma \ref{median lemma}(c),
$m^\gamma_u(B(x,r/2))\le m^{\gamma/C}_u(B_i)$. Since $\ph_i\ge C^{-1}$ on $B_i$, we have that 
$$m^\gamma_u(B(x,r/2))\le C u^{\gamma/C}_r(x).$$
The claim \eqref{comparability for median maximal functions} follows immediately from these estimates.
\end{proof}

\section{Capacity}\label{sec: capacity}
In this section, we define the $M^{s,X}$-capacity and prove some of its basic properties.

\begin{definition}
Let $0<s<\infty$. The $M^{s,X}$-capacity of a set $E\subset \cP$ is
\[
C_{M^{s,X}}(E)=\inf\Big\{\|u\|_{M^{s,X}}: u\in\mathcal A_{M^{s,X}}(E)\Big\},
\]
where 
\[
\mathcal A_{M^{s,X}}(E)=\big\{u\in M^{s,X}: u\ge 1 \text{ on a neighbourhood of } E\big\}
\] 
is a set of admissible functions for the capacity.
We say that a property holds \emph{quasieverywhere} if it holds outside a set of $M^{s,X}$-capacity zero.
\end{definition}

\begin{remark} 
Lemma \ref{max lemma} easily implies that
\[
C_{M^{s,X}}(E)=\inf\Big\{\|u\|_{M^{s,X}}: u\in\mathcal A_{M^{s,X}}'(E)\Big\},
\]
where
$\mathcal A_{M^{s,X}}'(E)=\{u\in \cA_{M^{s,X}}(E): 0\le u\le 1\}$. 
\end{remark}

\begin{remark}\label{cap remark 2} 
It is easy to see that the $M^{s,X}$-capacity is an outer capacity, which means that
\[
C_{M^{s,X}}(E)=\inf\big\{C_{M^{s,X}}(U): U\supset E,\ U \text{ is open}\big\}.
\]
\end{remark}

The $M^{s,X}$-capacity is not generally subadditive, but for most purposes, it suffices that it
satisfies inequality \eqref{eq: r-subadd} below.

\begin{lemma}\label{r-subadd} 
Let $0<s<\infty$ and let $\rho$ be the constant from \eqref{Aoki}.
Then 
\begin{equation}\label{eq: r-subadd}
C_{M^{s,X}}\big(\bigcup_{i\in I}E_i\big)^\rho\le 8\sum_{i\in I} C_{M^{s,X}}(E_i)^\rho
\end{equation}
whenever $E_i\subset \cP$, $i\in I\subset\n$.
\end{lemma}

\begin{proof} 
Let $\eps>0$. 
We may assume that $ \sum_{i\in I} C_{M^{s,X}}(E_i)^{\rho}<\infty$.
There are functions $u_i\in \cA'_{M^{s,X}}(E_i)$ with $g_i\in \cD^s(u_i)$ such that 
\[
\Big(\|u_i\|_{X}+\|g_i\|_{X}\Big)^{\rho}< C_{M^{s,X}}(E_i)^\rho+2^{-i}\eps.
\] 
By Lemma \ref{sup lemma},  $u=\sup_{i\in I} u_i\in \cA'_{M^{s,X}}(\cup_{i\in I} E_i)$ and 
$g=\sup_{i\in I}g_{i}\in\cD^s(u)$. Hence
\[
\begin{split}
C_{M^{s,X}}\Big(\bigcup_{i\in I} E_i\Big)^\rho&\le (\|u\|_X+\|g\|_X)^\rho\\
&\le \|u\|_X^\rho+\|g\|_X^\rho\\
&\le 4\sum\Big( \|u_i\|_X^\rho+\|g_i\|_X^\rho\Big)\\
&\le 8\sum\Big( \|u_i\|_X+\|g_i\|_X\Big)^\rho\\
&\le 8\Big(\eps+ \sum_{i\in I} C_{M^{s,X}}(E_i)^\rho\Big)
\end{split}
\]
and the claim follows by letting $\eps\to 0$. 
\end{proof}

A function $u:\cP\to[-\infty,\infty]$ is \emph{quasicontinuous}
if for every $\eps>0$, there exists a set $E\subset \cP$ such that $C_{M^{s,X}}(E)<\eps$ and the restriction of $u$ to $\cP\setminus E$ is continuous. By Remark \ref{cap remark 2}, the set $E$ can be chosen to be open.

The following lemma follows from a result of Kilpel\"ainen \cite{K}.
\begin{lemma}\label{Kil}
Suppose that $u$ and $v$ are quasicontinuous. If $u=v$ almost everywhere in an open set $U$, then $u=v$ quasieverywhere in $U$.
\end{lemma}

\begin{lemma}\label{quasicont rep} Suppose that $X$ has absolutely continuous quasi-norm. Then,
for every $u\in \dot M^{s,X}$, there exists quasicontinuous
$u^*$ such that $u=u^*$ almost everywhere.
\end{lemma}
\begin{proof} Suppose first that $u\in M^{s,X}$.
By Lemma \ref{cont lemma}, there are continuous functions $u_i\in M^{s,X}$ converging to $u$ in $M^{s,X}$ such that
\[
\|u_i-u_{i+1}\|_{M^{s,X}}<2^{-2i}
\]
for every $i\in\n$. Moreover, by \cite[Lemma 3.3]{CGO}, we may assume that $u_i\to u$ pointwise almost everywhere.
Denote $E_i=\{x\in\cP: |u_i(x)-u_{i+1}(x)|>2^{-i}\}$ and $F_j=\cup_{i=j}^\infty E_i$. Then
\[
|u_j-u_k|\le\sum_{i=j}^{k-1}|u_i-u_{i+1}|\le\sum_{i=j}^{k-1}2^{-i}=2^{1-j}
\]
in $\cP\setminus F_j$ for every $k>j$. Hence
$(u_i)$ converges pointwise in $\cP\setminus \cap_{j=1}^\infty F_j$
and the convergence is uniform in $\cP\setminus F_j$.
By continuity, $2^{i}|u_i-u_{i+1}|\in\cA_{M^{s,X}}(E_i)$ and so
\[
C_{M^{s,X}}(E_i)\le 2^i\|u_i-u_{i+1}\|_{M^{s,X}}<2^{-i}
\]
for every $i\in\n$.
Hence, by Lemma \ref{r-subadd},
\[
C_{M^{s,X}}(F_j)^\rho \le 8
\sum_{i=j}^\infty C_{M^{s,X}}(E_i)^\rho \le C2^{-j\rho},
\]
which implies that $C_{M^{s,X}}(\cap_{j=1}^\infty F_j)=0$. It follows
that the function  $u^*=\limsup_{i\to\infty} u_i$ is quasicontinuous.
Moreover, $u^*=\lim_{i\to\infty} u_i=u$ almost everywhere.

Suppose then that $u\in \dot M^{s,X}$. Let $x\in\cP$. For $k\in\n$, let $\ph_k$ be a Lipschitz
function of bounded support such that $\ph_k=1$ in $B(x,k)$. Then, by Lemma \ref{gradient for product}, $u_k=u\ph_k\in M^{s,X}$. By the first part of the proof, there exists quasicontinuous $u_k^*$ such that $u_k^*=u_k$ almost everywhere.
Since $u_{k+1}^*=u=u_k^*$ almost everywhere in $B(x,k)$, Lemma \ref{Kil}
implies that there exists $E_k$ with $C_{M^{s,X}}(E_k)=0$ such that $u_{k+1}^*=u_k^*$ in $B(x,k)\setminus E_k$. It follows that the limit
$\lim_{k\to\infty} u_k^*$ exists in $\cP\setminus E$, where, by Lemma \ref{r-subadd}, $E=\cup_{k=1}^\infty E_k$ is of $M^{s,X}$-capacity zero. 
Define $u^*=\limsup_{k\to\infty} u_k^*$. Then, clearly $u^*=u$ almost everywhere. Let $\eps>0$. For every $k\in\n$, there exists $U_k$ with $C_{M^{s,X}}(U_k)<2^{-k}\eps$ such that ${u_k^*}_{|\cP\setminus U_k}$ is continuous. It follows that $u^*_{|\cP\setminus \cup_{k=1}^\infty U_k \cup E}$ is continuous
and, by Lemma \ref{r-subadd}, $C_{M^{s,X}}(\cup_{k=1}^\infty U_k\cup E)<\eps$.
\end{proof}

The following lemma gives a useful characterization of the capacity in terms of quasicontinuous functions.
The proof of the lemma is essentially same as the proof of \cite[Theorem 3.4]{KKM}.
For $E\subset \cP$, denote
\[
\begin{split}
\mathcal{QA}_{M^{s,X}}(E)=\{&u\in M^{s,X}: u \text{ is quasicontinuous}\\ 
&\text{and }u\ge 1 \text{ quasieverywhere in } E\}
\end{split}
\]
and
\[
\widetilde C_{M^{s,X}}(E)=\inf_{u\in \mathcal{QA}_{M^{s,X}}(E)}\|u\|_{M^{s,X}}.
\]
\begin{lemma}\label{cap lemma}
Suppose that $X$ has absolutely continuous quasinorm. Then 
\[ 
\widetilde C_{M^{s,X}}(E)\le C_{M^{s,X}}(E)\le c_\Delta\,\widetilde C_{M^{s,X}}(E)
\]
for every $E\subset \cP$.
\end{lemma}
\begin{proof}
To prove the first inequality, let $u\in\cA_{M^{s,X}}(E)$ and let $u^*$ be a quasicontinuous representative of $u$. Since $u\ge 1$ in some open set $U$ containing $E$ and $u^*=u$ almost everywhere, it follows that $\min\{0,u^*-1\}= 0$ almost everywhere in $U$. Since $\min\{0,u^*-1\}$ is quasicontinuous, the
equality actually holds quasieverywhere in $U$. Hence $u^*\ge 1$ quasieverywhere in $U$, which implies that $u^*\in\mathcal{QA}_{M^{s,X}}(E)$.

For the second inequality, let $v\in\mathcal{QA}_{M^{s,X}}(E)$. By truncation, we may assume that $0\le u\le 1$.
 Fix $0 < \eps < 1$, and choose an open set $V$ with $C_{M^{s,X}}(V) < \eps$ so that
$v = 1$ on $E \setminus V $ and that $v$ is continuous in $\cP\setminus V$.
By continuity, there is an open set
$U \subset \cP$ such that 
\[
\{ x \in \cP : v(x) > 1-\eps \}\setminus V = U \setminus V.
\]
Clearly, $E \setminus V \subset U \setminus V$.
Choose $u \in \cA_{M^{s,X}}(V)$ such that $\|u\|_{M^{s,X}} < \eps$
and that $0 \le u \le 1$.
Define $w=v/(1-\eps)+u$. Then $w\ge 1$ on $(U\setminus V)\cup V=U\cup V$ which is an open neighbourhood of $E$. Hence $w\in \cA_{M^{s,X}}(E)$ and so
\[
C_{M^{s,X}}(E)\le \|w\|_{M^{s,X}}\le c_\Delta\Big(\frac{1}{1-\eps}\|v\|_{M^{s,X}}+\|u\|_{M^{s,X}}\Big)
\le c_\Delta\Big(\frac{1}{1-\eps}\|v\|_{M^{s,X}}+\eps\Big).
\]
Since $\eps>0$ and $v\in\mathcal{QA}_{M^{s,X}}(E)$ are arbitrary, the desired inequality $C_{M^{s,X}}(E)\le c_\Delta\,\widetilde C_{M^{s,X}}(E)$ follows.
\end{proof}




\section{Generalized Lebesgue points}\label{sec: qc}
In this section, we prove the first main result of the paper, Theorem \ref{main thm}. 
The main ingredient of the proof of is a capacitary weak type estimate, Theorem \ref{cap weak type} below. 


\begin{lemma}\label{gradient for discrete median maximal function} 
Let $0<s\le 1$, $0<\gamma\le 1/2$ and $0<R<\infty$.
Let $u\in L^0$ and $g\in\cD^s(u)\cap L^0$. 
Then there exists a constant $C\ge 1$ such that $C\M^{\gamma/C}_{3R}g$ is an $s$-gradient of $\M^{\gamma,*}_R u$.

\end{lemma}

\begin{proof} 
Let $r>0$. By the definition of the discrete $\gamma$-median convolution $u^\gamma_r$ and by the properties of the functions $\ph_i$,
\[
u^\gamma_r=u+\sum_{i\in I}(m^\gamma_u(B_i)-u)\ph_i.
\]
By Lemma \ref{gradient for product}, function
\[
(g+Cr^{-s}|u-m^\gamma_u(B_i)|)\ch{2B_i}
\]
is an $s$-gradient of function $(m^\gamma_u(B_i)-u)\varphi_i$ for each $i$.

Let $x\in 2B_i\setminus E$, where $E$ is the exceptional set for \eqref{eq: gradient}. 
Using Lemma \ref{median lemma} and the facts that $B_i\subset B(x,3r)$ and $\mu(B(x,3r))\le C\mu(B_i)$, we obtain
\[
\begin{split}
|u(x)-m_u^{\gamma}(B_i)|&\le m_{|u-u(x)|}^{\gamma}(B_i)
\le Cr^s\big( m_{g}^{\gamma}(B_i)+g(x)\big)\\
&\le Cr^s\big( m_{g}^{\gamma/C}(B(x,3r))+g(x)\big)\\
&\le Cr^s(g(x)+\M^{\gamma/C}_{3r} g(x)).
\end{split}
\]
Since $g(x)\le\M^{\gamma/C}_{3r}g(x)$ for almost every $x$ and since the balls $2B_i$ have bounded overlap, it follows that
\[
C\M^{\gamma/C}_{3r} g\in \cD^s(u^\gamma_r),
\]
for every $r>0$. Consequently, by Lemma \ref{sup lemma},
\[
C\M^{\gamma/C}_{3R} g\in \cD^s(\M^{\gamma,*}_R u)
\]
as desired.
\end{proof}

\begin{theorem}\label{cap weak type}
Suppose that the assumptions of Theorem \ref{main thm} are in force.
Then, for every ball $B=B(x_0,r)$ and for every $0<\gamma< 1$, there exists a constant
$C$ such that
\[
C_{M^{s,X}}\big(\{x\in B: \M^{\gamma,*}_{1/3} u(x) >\lambda\}\big)
\le C\lambda^{-1}\|u\|_{M^{s,X}}
\]
for every $u\in M^{s,X}$ and $\lambda>0$.
\end{theorem}

\begin{proof} Since $\M^{\gamma,*}_{1/3} u\le \M^{\gamma',*}_{1/3} u$ when $\gamma'\le\gamma$, it suffices to prove the claim for $0<\gamma\le 1/2$.
Let $u\in M^{s,X}$ and let $\varphi:\cP\to [0,1]$ be a Lipschitz function such that $\varphi=1$ in $B(x_0,r+4/3)$ and $\ph=0$ outside $B(x_0,r+2)$. Then
$\M^{\gamma,*}_{1/3} u=\M^{\gamma,*}_{1/3}(u\ph)$ in $B$ and so 
\[
\{x\in B: \M^{\gamma,*}_{1/3} u(x) >\lambda\}\subset \{x\in \cP: \M^{\gamma,*}_{1/3} (u\ph)(x) >\lambda\}.
\]
By \eqref{comparability for median maximal functions} and \eqref{M gamma bd},
 \[
 \begin{split}
     \|\M^{\gamma,*}_{1/3} (u\ph)\|_X\le C\|\M^{\gamma/C}_{1} (u\ph)\|_X\le C\|(\M^{\gamma/C}_{1}u)\chi_{B(x_0,r+3)}\|_X\le C\|u\|_X
 \end{split}
 \]
and, by Lemma \ref{gradient for discrete median maximal function},
Lemma \ref{gradient for product} and \eqref{M gamma bd},
 \[
 \begin{split}
     \|\M^{\gamma,*}_{1/3} (u\ph)\|_{\dot M^{s,X}}
     &\le C\inf_{g\in\cD^s(u\ph)}\|\M^{\gamma/C}_{1}g\|_X
     \le C\inf_{g\in\cD^s(u)}\|\M^{\gamma/C}_{1}((u+g)\chi_{B(x_0,r+2)})\|_X\\
     &\le C\inf_{g\in\cD^s(u)}\|(\M^{\gamma/C}_{1}(u+g))\chi_{B(x_0,r+3)})\|_X\\
     &\le C\inf_{g\in\cD^s(u)}\|u+g\|_X\le C\|u\|_{M^{s,X}}.
 \end{split}
 \]
 Hence, $\M^{\gamma,*}_{1/3} (u\ph)\in M^{s,
 X}$. Since $\M^{\gamma,*}_{1/3} (u\ph)$ is lower semicontinuous, $\lambda^{-1}\M^{\gamma,*}_{1/3} (u\ph)\in\cA_{M^{s,X}}(\{x\in\cP: M^{\gamma,*}_{1/3} (u\ph)>\lambda\})$. Thus,
 \[
 \begin{split}
     C_{M^{s,X}}(\{x\in\cP: \M^{\gamma,*}_{1/3} (u\ph)>\lambda\})\le \lambda^{-1}\|\M^{\gamma,*}_{1/3} (u\ph)\|_{M^{s,X}}
     \le C\lambda^{-1}\|u\|_{M^{s,X}}
 \end{split}
 \]
 and the claim follows.
\end{proof}

\begin{lemma}\label{cap limsup implies gen leb points}
Suppose that $X$ has absolutely continuous quasinorm and that $B$ is a ball such that, for every $0<\gamma<1$ and $\lambda>0,$
\begin{equation}\label{cap limsup}
   \lim_{k\to \infty} C_{M^{s,X}}(\{x\in B: \limsup_{r\to 0}m^\gamma_{|u_k|}(B(x,r))>\lambda\})=0.
\end{equation}
whenever $\lim_{k\to\infty}\|u_k\|_{M^{s,X}}=0$. Then, for every quasicontinuous $u\in M^{s,X}$, quasievery point in $B$ is a generalized Lebesgue point of $u$.

\end{lemma}

\begin{proof} 
By Lemma \ref{cont lemma}, continuous
functions are dense in $M^{s,X}$. Let $u\in M^{s,X}$ be quasicontinuous and
let $v_k\in M^{s,X}$, $k=1,2,\dots,$ be continuous such that \[\|u-v_k\|_{M^{s,X}}\to 0\] as $k\to\infty$. Denote $w_k=u-v_k$. 
Fix $0<\gamma<1$ and $\lambda>0$. 
By Lemma \ref{median lemma},
\[
\begin{split}
\limsup_{r\to 0}m^{\gamma}_{|u-u(x)|}(B(x,r))&\le \limsup_{r\to 0}m^{\gamma/2}_{|v_k-v_k(x)|}(B(x,r))+\limsup_{r\to 0}m^{\gamma/2}_{|w_k-w_k(x)|}(B(x,r))\\
&\le \limsup_{r\to 0}m^{\gamma/2}_{|w_k|}(B(x,r))+|w_k(x)|.
\end{split}
\]
Hence, by Lemma \ref{r-subadd}, 
\[
\begin{split}
&C_{M^{s,X}}(\{x\in B:\limsup_{r\to 0}m^{\gamma}_{|u-u(x)|}(B(x,r))>\lambda\})^\rho
\\&\le 8\,C_{M^{s,X}}(\{x\in B:\limsup_{r\to 0}m^{\gamma/2}_{|w_k|}(B(x,r))>\lambda/2\})^\rho\\
&+8\,C_{M^{s,X}}(\{x\in B:|w_k(x)|>\lambda/2\})^\rho.
\end{split}
\]
By assumption, 
\[
\begin{split}
C_{M^{s,X}}(\{x\in B:\limsup_{r\to 0}m^{\gamma/2}_{|w_k|}(B(x,r))>\lambda/2\})\to 0
\end{split}
\]
as $k\to\infty$.
Since $|w_k|$ is quasicontinuous, Lemma \ref{cap lemma} gives
\[
\begin{split}
C_{M^{s,X}}(\{x\in B:|w_k(x)|>\lambda/2\}) &\le C\widetilde C_{M^{s,X}}(\{x\in B:|w_k(x)|>\lambda/2\})\\
&\le C2\lambda^{-1}\|w_k\|_{M^{s,X}}\to 0
\end{split}
\]
as $k\to\infty$.
It follows that
\[
C_{M^{s,X}}(\{x\in B:\limsup_{r\to 0}m^{\gamma}_{|u-u(x)|}(B(x,r))>\lambda\})=0
\]
for every $0<\gamma<1$ and $\lambda>0$. 
Denote
\[
\begin{split}
E=\{x\in B:\limsup_{r\to 0}m^{\gamma}_{|u-u(x)|}(B(x,r))>0 \text{ for some $0<\gamma<1$}\}.
\end{split}
\]
Then
\[
\begin{split}
E=\bigcup_{n,m\ge 2}\{x\in B:\limsup_{r\to 0}m^{1/m}_{|u-u(x)|}(B(x,r))>1/n\}
\end{split}
\]
and so, by
Lemma \ref{r-subadd}, $C_{M^{s,X}}(E)=0$.
Since, by Lemma \ref{median lemma},
\[
| m^\gamma_{u}(B(x,r))-u(x)|
=| m^\gamma_{u-u(x)}(B(x,r))|
\le m^{\min\{\gamma,1-\gamma\}}_{|u-u(x)|}(B(x,r)),
\]
the claim follows.
\end{proof}

\begin{proof}[Proof of Theorem \ref{main thm}]
Let $u\in \dot M^{s,X}$ be quasicontinuous. Fix $x\in\cP$ and, for $k\in\n$, let $\ph_k:\cP\to [0,1]$ be a Lipschitz
function of bounded support such that $\ph_k=1$ in $B(x,k)$. Then, by Lemma \ref{gradient for product}, $u_k=u\ph_k\in M^{s,X}$.
By Theorem \ref{cap weak type}, Lemma \ref{comparability} and Lemma \ref{cap limsup implies gen leb points}, for every $k$, quasievery point is a generalized Lebesgue point of $u_k$. Hence, for every $k\in\n$,
quasievery point in $B(x,k)$ is a generalized Lebesgue point of $u$.
Thus, by Lemma \ref{r-subadd}, quasievery point is a generalized Lebesgue point of $u$.
\end{proof}

\section{Lebesgue points}\label{sec: integral averages}
In this section we prove Theorem \ref{main thm 2}.
\begin{lemma}\label{gradient for discrete maximal function} 
Suppose that $\mu$ satisfies \eqref{Q}. Let $0<s\le 1$, $R>0$, $u\in L^0$ and $g\in\cD^s(u)\cap L^q_\text{loc}$ where $q=Q/(Q+s)$. 
Then there exists a constant $C$ such that $C(\M_{6R }g^q)^{1/q}$ is an $s$-gradient of  $\M^*_R u$.
\end{lemma}

\begin{proof} 
By the Sobolev--Poincar\'e inequality (\cite[Lemma 2.2]{GKZ}), $u$ is locally integrable and there exists a constant $C$ such that
\begin{equation}\label{S-P}
    \inf_{c\in\re}|u-c|_{B(x,r)}\le Cr^s ((g^q)_{B(x,2r)})^{1/q}
\end{equation}
for every $x\in\cP$ and $r>0$.

Fix $r>0$.
By the definition of the discrete convolution $u_r$ and by the properties of the functions $\ph_i$, 
\[
u_r=u+\sum_{i\in\n}(u_{B_i}-u)\ph_i.
\]
By Lemma \ref{gradient for product}, function
\[
(g+Cr^{-s}|u-u_{B_i}|)\ch{2B_i}
\]
is an $s$-gradient of the function $(u_{B_i}-u)\varphi_i$ for each $i$.
A standard chaining argument and \eqref{S-P} imply that, for almost every $x\in 2B_i$ ,
\[
|u(x)-u_{B_i}|\le C(\M_{6r}g^q(x))^{1/q}.
\]
Since $g(x)\le(\M_{6r}g^q(x))^{1/q}$, for almost every $x$, and since the balls $2B_i$ have bounded overlap, it follows that
\[
C(\M_{6r} g^q)^{1/q}\in \cD^s(u_r).
\]
The claim follows by Lemma \ref{sup lemma}.
 \end{proof}

\begin{lemma}\label{uusi lemma} 
Let $u\in\dot M^{s,X}$ be such that $u=0$ outside $B=B(x_0,r)$ and suppose that there exists $y_0\in \cP$ such that $d(x_0,y_0)=3r$.
Then there exists a constant $C$ such that 
$\|u\|_{M^{s,X}(\cP)}\le C\inf_{g\in \cD^s(u)}\|g\chi_{4B}\|_X$.
\end{lemma}
\begin{proof}
Let $g\in \cD^s(u)$. Then $g\chi_{2B}+r^{-s}u\in\cD^s(u)$. Indeed, if $x,y\in 2B$, then, by definition,
\[
|u(x)-u(y)|\le d(x,y)^s(g(x)+g(y)),
\]
and if one of the points, say $x$, does not belong to $2B$, then 
\[
|u(x)-u(y)|=|u(y)|\le d(x,y)^sr^{-s}|u(y)|.
\]
Hence, it suffices to show that $\|u\|_{X}\le C\|g\chi_{4B}\|_{X}.$
For almost every $x\in B$,
\[
|u(x)|\le 5^sr^s\left(g(x)+\essinf_{y\in 4B\setminus 2B}g(y)\right).
\]
Since $B(y_0,r)\subset 4B\setminus 2B$, we have that $\mu(4B\setminus 2B)>0$.
Clearly, 
\[
(\essinf_{4B\setminus 2B}g)\|\chi_{4B\setminus 2B}\|_X\le \|g\chi_{4B}\|_X.
\]
Hence
\[
\begin{split}
\|u\|_{X}&=\|u\chi_B\|_{X}\\
&\le C\Big(\|g\chi_B\|_{X}+(\essinf_{4B\setminus 2B}g)\|\chi_B\|_X\Big)\\
&\le C\Big(\|g\chi_B\|_{X}+\frac{\|\chi_B\|_X}{\|\chi_{4B\setminus 2B}\|_X}\|g\chi_{4B}\|_X\Big)\\
&\le C\|g\chi_{4B}\|_{X}.
\end{split}
\]
\end{proof}

\begin{theorem}\label{cap weak type 2}
 Suppose that the assumptions of Theorem \ref{main thm 2} are in force.
Let $B=B(x_0,r)$ be a ball and assume that the sphere $\{y: d(y,x_0)=6r\}$ is nonempty.
Then there exist constants $C\ge 1$ and $R>0$ such that
\[
C_{M^{s,X}}\big(\{x\in B: \M_{R}^* u(x) >\lambda\}\big)
\le C\lambda^{-1}\|u\|_{M^{s,X}}
\]
for every $u\in M^{s,X}$ and $\lambda>0$.
\end{theorem}

\begin{proof}
Let $R=\frac19\min\{1,r\}$ and let $\varphi:\cP\to [0,1]$ be a Lipschitz function such that $\varphi=1$ in $B(x_0,r+4R)$ and $\varphi=0$ outside $B(x_0,r+5R)$.
Then, by 
Lemma \ref{gradient for product},
 $u\varphi\in M^{s,X}$ 
and 
\begin{equation}\label{est}
\|u\ph\|_{M^{s,X}}\le C\|u\|_{M^{s,X}}.
\end{equation}
If $x\in B$, then $\M^*_Ru(x)=\M^*_R(u\ph)(x)$.  Hence,
\[
\{x\in B: \M^*_R u(x) >\lambda\}\subset \{x\in \cP: \lambda^{-1}\M^*_R (u\ph)(x) >1\}.
\]
If $x\in \cP\setminus B(x_0,r+9R)$, then $\M^*_r(u\ph)(x)=0$.
In particular, $\M^*_r(u\ph)=0$ outside $2B$.
Hence, by Lemma \ref{gradient for discrete maximal function} and Lemma \ref{uusi lemma},  $\M^{*}_R (u\ph)\in M^{s,X}$ and
\[
\|\M^{*}_R (u\ph)\|_{M^{s,X}}\le C\inf_{g\in\cD^s(\M^*_R(u\ph))}\|g\chi_{8B}\|_{X}
\le C\inf_{g\in\cD^s(u\ph)}\|(\M_{6R}g^q)^{1/q}\chi_{8B}\|_{X},
\]
where $q=Q/(Q+s)$.
Since $\M^{*}_R (u\ph)$ is lower semicontinuous, it follows that
\[
\begin{split}
C_{M^{s,X}}\big(\{x\in B: \M^{*}_{R} u(x) >\lambda\}\big)
&\le \lambda^{-1}\|\M^{*}_R (u\ph)\|_{M^{s,X}}\\
&\le C\lambda^{-1}\inf_{g\in\cD^s(u\ph)}\|(\M_{6R}g^q)^{1/q}\chi_{8B}\|_{X}.
\end{split}
\]
The fact that $6R<1$, assumption \eqref{M q-bd} and \eqref{est} imply that
\[
\begin{split}
\inf_{g\in\cD^s(u\ph)}\|(\M_{6R}g^q)^{1/q}\chi_{8B}\|_{X}
&\le \inf_{g\in\cD^s(u\ph)}\|(\M_{1}g^q)^{1/q}\chi_{8B}\|_{X}\\
&\le \inf_{g\in\cD^s(u\ph)}\|g\|_{X}\\
&\le \|u\varphi\|_{M^{s,X}}\\
&\le C\|u\|_{M^{s,X}}\\
\end{split}
\]
and the claim follows.
\end{proof}

The proof of the next lemma, which is very similar to the proof of Lemma \ref{cap limsup implies gen leb points}, will be omitted.

\begin{lemma}\label{cap limsup implies leb points}
Suppose that $\|\cdot\|_X$ is absolutely continuous and that $B\subset\cP$ is ball such that, for every $\lambda>0$,
\begin{equation}\label{cap limsup 2}
   \lim_{i\to \infty} C_{M^{s,X}}(\{x\in B: \limsup_{r\to 0}|u_i|_{B(x,r)}>\lambda\})=0
\end{equation}
whenever $\lim_{i\to\infty}\|u_i\|_{M^{s,X}}=0$. Then, for every quasicontinuous $u\in M^{s,X}$, quasievery point in $B$ is a Lebesgue point of $u$.
\end{lemma}

\begin{proof}[Proof of Theorem \ref{main thm 2}]
We may assume that $\cP$ contains at least two points. Then $\cP$ can be covered by balls $B_k=B(x_k,r_k)$, $k\in I$, where $I\subset\n$, such that the spheres $\{y: d(x_k,y)=6r_k\}$ are nonempty. 

Let $u\in\dot M^{s,X}$ be quasicontinuous and, for $k\in I$, let $\ph_k$ be a Lipschitz
function of bounded support such that $\ph_k=1$ in $B(x_k,r_k)$. Then, by Lemma \ref{gradient for product}, $u_k=u\ph_k\in M^{s,X}$.
By Theorem \ref{cap weak type 2}, Lemma \ref{comparability} and Lemma \ref{cap limsup implies leb points}, for every $k\in I$, quasievery point in $B_k$ is a Lebesgue point of $u_k$ and hence of $u$. Hence,
by Lemma \ref{r-subadd}, quasievery point in $\cP$ is a Lebesgue point of $u$.
\end{proof}

\section{Proof of Theorem \ref{main thm 3}}

If $\alpha_X<\infty$, then clearly $\M^\gamma$ is bounded on $X$. Hence, the first part of Theorem \ref{main thm 3} follows from Theorem \ref{main thm}. The second part follows from Theorem \ref{main thm 2} via the following lemma.

\begin{lemma}\label{index} Let $p\le 1$. If $\alpha_X<1/p$, then the operator $u\mapsto (\M |u|^p)^{1/p}$ is bounded on $X$.
\end{lemma}
\begin{proof}
By assumption, there are constants $\alpha<1/p$ and $C>0$ such that
\begin{equation}\label{est1}
\|\M^\gamma u\|_X\le C\gamma^{-\alpha}\|u\|_X.
\end{equation}
for every $u\in X$ and $0<\gamma<1$.
Denote by $v^*$ the decreasing rearrangement of a function $v$, that is,
\[
v^*(t)=\inf\big\{a\ge 0: \mu(\{x\in \cP: |v(x)|>a\})< t\}.
\]
Then, for every ball $B$, we have
\[
\vint{B} |u|^p\,d\mu=\frac{1}{\mu(B)}\int_0^{\mu(B)}(u\chi_B)^*(t)^p\,dt=\int_0^1 (u\chi_B)^*(\gamma\mu(B))^p\,d\gamma=\int_0^1 m^\gamma_{|u|}(B)^p\,d\gamma.
\]
Hence,
\[
M|u|^p(x)\le\int_0^1 \M^\gamma u(x)^p\,d\gamma\le\sum_{i=1}^\infty 2^{-i}\M^{2^{-i}}u(x)^p.
\]
Let $0<\eps<1$ be such that $\alpha<\eps/p$. By the H\"older inequality, we obtain
\[
(M|u|^p(x))^{1/p}\le C\sum_{i=1}^\infty 2^{-i\eps/p}\M^{2^{-i}}u(x).
\]
Thus, by \eqref{est1} and \eqref{Aoki},
\[
\begin{split}
\|(M|u|^p)^{1/p}\|_X&\le C\|\sum_{i=1}^\infty 2^{-i\eps/p}\M^{2^{-i}}u\|_X\le C(\sum_{i=1}^\infty\| 2^{-i\eps/p}\M^{2^{-i}}u\|_X^{\sigma})^{1/\sigma}\\
&\le C(\sum_{i=1}^\infty  2^{-i(\eps/p-\alpha)\sigma})^{1/\sigma}\|u\|_X\le C\|u\|_X.
\end{split}
\]
\end{proof}

\section{Examples}
\subsection{Lorentz spaces} For $0<p<\infty$, $0<q\le\infty$ and $u\in L^0$, denote 
\[
\|u\|_{L^{p,q}}=
\left(\int_0^\infty \lambda^q\mu(\{x\in\cP:|u(x)|>\lambda\})^{q/p}\,\frac{d\lambda}{\lambda}\right)^{1/q}
\]
when $0<q<\infty$ and
\[
\|u\|_{L^{p,\infty}}=\sup_{\lambda>0}\ \lambda\,\mu(\{ x\in\cP:|u(x)|>\lambda\})^{1/p}.
\]
Then $L^{p,q}=\{u\in L^0: \|u\|_{L^{p,q}}<\infty\}$ equipped with $\|\cdot\|_{L^{p,q}}$ is a quasi-Banach function space.
If $0<p,q<\infty$, quasinorm $\|\cdot\|_{L^{p,q}}$ is absolutely continuous.

\begin{lemma}\label{Mgamma ie} There exists a constant $C=C(c_d)$ such that
\begin{equation}
\mu(\{x\in\cP: \M^\gamma u(x)>\lambda\})\le C \gamma^{-1}\mu(\{x\in\cP: |u(x)|>\lambda\}) 
\end{equation}
for every $u\in L^0$ and $\lambda>0$.
\end{lemma}
\begin{proof}
It follows easily from the definitions that,
for every $u\in L^0$, $x\in \cP$, $0<\gamma<1$ and $\lambda>0$, we have
\[
\M^\gamma u(x)>\lambda\iff \M \chi_{\{y\in\cP:|u(y)|>\lambda\}}(x)>\gamma.
\]
Hence, the claim follows from the well-known weak type inequality 
for the Hardy--Littlewood maximal operator.
\end{proof}

The following lemma is an immediate consequence of Lemma \ref{Mgamma ie}.

\begin{lemma}\label{index for Lorentz}
Let $0<p<\infty$ and $0<q\le\infty$. There exists a constant $C=C(c_d,p)$ such that
\[
\|\M^\gamma u\|_{L^{p,q}}\le C\gamma^{-1/p}\|u\|_{L^{p,q}}
\]
for every $u\in L^{p,q}$ and $0<\gamma<1$. Consequently, $\alpha_{L^{p,q}}\le 1/p$.
\end{lemma}

Lemma \ref{index for Lorentz} and Theorem \ref{main thm 3} imply the following result for Lorentz--Haj\l asz functions.

\begin{theorem}\label{Lorentz 1}
Let $0<s\le 1$, $0<p,q<\infty$ and let
$u\in\dot M^{s,L^{p,q}}$ be quasicontinuous. Then quasievery point is a generalized Lebesgue point of $u$.
Moreover, if $p>Q/(Q+s)$, then quasievery point is a Lebesgue point of $u$.
\end{theorem}

\subsection{Orlicz spaces} 
Let $\Phi:[0,\infty)\to[0,\infty)$ be a continuous increasing function and suppose that there exists a constant $C\ge 1$ such that 
\begin{equation}\label{Phi 1}
\Phi(t/C)\le \Phi(t)/2
\end{equation}
for every $t>0$. Then
\[
L^\Phi=\{u\in L^0: \int_\cP\Phi(|u|/\lambda)d\mu<\infty \text{ for some }\lambda>0\}
\]
equipped with a quasi-norm
\[
\|u\|_{L^\Phi}=\inf\{\lambda>0: \int_\cP\Phi(| u|/\lambda)d\mu\le 1 \}
\]
becomes a quasi-Banach function space.

It is easy to see that \eqref{Phi 1} is equivalent to the existence of constants $C_0\ge 1$ and $\beta>0$ such that
\begin{equation}\label{Phi 2}
\Phi(s)\le C_0(s/t)^\beta \Phi(t)
\end{equation}
whenever $0<s\le t<\infty$.

A function $\Phi:[0,\infty)\to[0,\infty)$ is doubling, if there exists a constant $C$ such that
\begin{equation}\label{Phi doubling}
\Phi(2t)\le C\Phi(t)
\end{equation}
for every $t>0$. It is easy to see that if $\Phi$ is doubling, then $\|\cdot\|_{L^\phi}$ is absolutely continuous.

\begin{lemma}\label{Orlicz lemma 1}
Let $\Phi$ satisfy \eqref{Phi 2}. Then there exists a constant $C$ such that
\[
\|\M^\gamma u\|_{L^\Phi}\le C\gamma^{-1/\beta}\|u\|_{L^\Phi}
\]
for every $u\in L^\Phi$ and $0<\gamma<1$. Consequently, $\alpha_{L^\Phi}\le 1/\beta$.
\end{lemma}

\begin{proof}
We may assume $\|u\|_{L^\Phi}\le 1$. Since $\Phi(\M^\gamma u)=\M^\gamma\Phi(u)$ and, by Lemma \ref{Mgamma ie},
$\|\M^\gamma v\|_{L^1}\le C_1\gamma^{-1}\|v\|_{L^1}$ for all $v\in L^1$, we have that
\[
\begin{split}
    \|\Phi((\M^\gamma u)/\lambda)\|_{L^1}&\le C_0\lambda^{-\beta} \|\Phi(\M^\gamma u)\|_{L^1}=  C_0\lambda^{-\beta}\|M^\gamma \Phi(u)\|_{L^1}\\
    &\le C_0C_1\lambda^{-\beta}\gamma^{-1}\|\Phi(u)\|_{L^1}\le C_0C_1\lambda^{-\beta}\gamma^{-1}\le 1
\end{split}
\]
when $\lambda\ge (C_0C_1)^{1/\beta}\gamma^{-1/\beta}$.
\end{proof}

By combining Theorem \ref{main thm 3} and Lemma \ref{Orlicz lemma 1}, we obtain the following result
for Orlicz--Haj\l asz spaces $M^{s,\Phi}:=M^{s,L^\Phi}$.

\begin{theorem}\label{Orlicz 1}
Suppose that $\mu$ satisfies \eqref{Q} and that $\Phi$ is doubling and satisfies \eqref{Phi 1}. Let $0<s\le 1$ and let $u\in\dot M^{s,\Phi}$ be quasicontinuous. Then quasievery point is a generalized Lebesgue point of $u$.
Moreover, if $\Phi$ satisfies \eqref{Phi 2} with $\beta>Q/(Q+s)$, then quasievery point
is a Lebesgue point of $u$.
\end{theorem}

\subsection{Variable exponent spaces} Let $p:\cP\to(0,\infty)$ be a measurable function. The space $L^{p(\cdot)}$ consisting of functions $u$ for which
\[
\int_\cP(|u(x)|/\lambda)^{p(x)}d\mu(x)<\infty
\]
for some $\lambda>0$ equipped with quasinorm
\[
\|u\|_{L^{p(\cdot)}}=\inf\{\lambda>0: \int_\cP(|u(x)|/\lambda)^{p(x)}d\mu(x)\le 1\}
\]
is a quasi-Banach function space.

A measurable function $p:\cP\to(0,\infty)$ is \emph{locally log-H\"older continuous} if
there exists a constant $C_p>0$ such that
\begin{equation}\label{log holder 1}
|p(x)-p(y)|\le C_{p}/\log(e+1/d(x,y))
\end{equation}
for all $x,y\in\cP$ 

Denote
\[
p_-=\essinf_{x\in\cP}p(x) \ \ \text{ and } \ \  p_+=\esssup_{x\in\cP}p(x)
\]

We need the following result from \cite{LP}.
\begin{lemma}\label{median max op bounded}
Let $p:\cP\to(0,\infty)$ be locally log-H\"older continuous with $p_->0$ and $p_+<\infty$. Suppose that
there exist constants $p_\infty>0$ and $0<a<1$ such that
\begin{equation}\label{decay}
\int_\cP a^{1/|p(x)-p_\infty|}d\mu(x)<\infty.
\end{equation}
Then there exists a constant $C$ such that
\[
\|\M^\gamma u\|_{L^{p(\cdot)}} \le C\gamma^{-1/p_-} \|u\|_{L^{p(\cdot)}} 
\]
for every $u\in L^{p(\cdot)}$ and $0<\gamma<1$.
Consequently, $\alpha_{L^{p(\cdot)}}\le 1/p_-$.
\end{lemma}

The next lemma follows from \cite[Lemma 2.3]{AHH}, \cite[Corollary 3.5]{AHH} and from the fact that
the function $t\mapsto 1/t$ is bi-Lipschitz from $[a,b]$ to $[1/b,1/a]$ whenever $0<a<b<\infty$.

\begin{lemma}\label{ext lemma}
Suppose that $p:\cP\to(0,\infty)$ is locally log-H\"older continuous and that $p_->0$ and $p_+<\infty$. 
Then, for any ball $B\subset\cP$, there exists a locally log-H\"older continuous extension $\tilde p$ of $p_{|B}$
such that $\tilde p_-=p_-$, $\tilde p_+=p_+$ and 
\begin{equation}\label{decay for tilde p}
\int_\cP a^{1/|\tilde p(x)-\tilde p_\infty|}d\mu(x)<\infty.
\end{equation}
for some $\tilde p_\infty>0$ and $0<a<1$.
\end{lemma}

\begin{theorem}\label{variable 1}
Suppose that $\mu$ satisfies \eqref{Q} and that  $p:\cP\to(0,\infty)$ is locally log-H\"older continuous with $p_+<\infty$.  Let $0<s\le 1$ and let
$u\in\dot M^{s,p(\cdot)}:=M^{s,L^{p(\cdot)}}$ be quasicontinuous.
\begin{itemize}
    \item[(1)] If $p_->0$, then quasievery point is a generalized Lebesgue point of $u$.
    \item[(2)] If $p_->Q/(Q+s)$, then quasievery point is a Lebesgue point of $u$.
\end{itemize}
\end{theorem}

\begin{proof}
(1) Since $p_+<\infty$, $L^{p(\cdot)}$ has absolutely continuous quasinorm. 
By Theorem \ref{main thm}, it suffices to show that, for every ball $B=B(x,r)$ and $0<\gamma<1$, there exists a constant $C$ such that $$\|(\M^\gamma_1 g)\chi_B\|_{L^{p(\cdot)}}\le C\|g\|_{L^{p(\cdot)}}$$ for every $g\in L^{p(\cdot)}$.
By Lemma \ref{ext lemma}, $p_{|B(x,r+1)}$ can be extended to $\tilde p$ on $\cP$ such that $\tilde p_+<\infty$, $\tilde p_->0$ and 
\eqref{decay for tilde p} holds true
for some $0<a<1$ and $\tilde p_\infty>0$.
By Lemma \ref{median max op bounded},
there exists a constant $C$ such that
\[
\|\M^\gamma v\|_{L^{\tilde p(\cdot)}} \le C\gamma^{-1/\tilde p_-} \|v\|_{L^{\tilde p(\cdot)}} 
\]
for every $v\in L^{\tilde p(\cdot)}$.
 Thus,
\[
\begin{split}
\|(\M^\gamma_1 g)\chi_B\|_{L^{p(\cdot)}}&\le\|\M^\gamma(g\chi_{B(x,r+1)})\|_{L^{\tilde p(\cdot)}}\\
&\le C\gamma^{-1/\tilde p_-}\|g\chi_{B(x,r+1)}\|_{L^{\tilde p(\cdot)}}\\
&\le C\gamma^{-1/\tilde p_-}\|g\|_{L^{p(\cdot)}}
\end{split}
\]
for every $g\in L^{p(\cdot)}$.

(2) Denote $q=Q/(Q+s)$. By Theorem \ref{main thm 2}, it suffices to show that, for every ball $B=B(x,r)$, there exists a constant $C$ such that 
\[
\|(\M_1 g^q)^{1/q}\chi_B\|_{L^{p(\cdot)}}\le C\|g\|_{L^{p(\cdot)}}
\]
for every $0\le g\in L^{p(\cdot)}.$
By Lemma \ref{ext lemma}, $p_{|B(x,r+1)}$ can be extended to $\tilde p$ on $\cP$ such that $\tilde p_+<\infty$, $\tilde p_->q$ and 
\eqref{decay for tilde p} holds true
for some $0<a<1$ and $\tilde p_\infty>0$.
Hence, by Lemma \ref{median max op bounded}, $\alpha_{L^{\tilde p(\cdot)}}\le 1/\tilde p_-<1/q$ and so, by Lemma \ref{index}, operator $g\mapsto(\M |g|^q)^{1/q}$ is bounded on $L^{\tilde p(\cdot)}$.
Thus,
\[
\begin{split}
\|(\M_1 g^q)^{1/q}\chi_B\|_{L^{p(\cdot)}}&\le\|(\M(g\chi_{B(x,r+1)})^q)^{1/q}\|_{L^{\tilde p(\cdot)}}\le C\|g\chi_{B(x,r+1)}\|_{L^{\tilde p(\cdot)}}\le C\|g\|_{L^{p(\cdot)}}
\end{split}
\]
for every $0\le g\in L^{p(\cdot)}$.
\end{proof}



\vspace{0.5cm}
\noindent
\small{\textsc{T.H.},}
\small{\textsc{Department of Mathematics},}
\small{\textsc{P.O. Box 11100},}
\small{\textsc{FI-00076 Aalto University},}
\small{\textsc{Finland}}\\
\footnotesize{\texttt{toni.heikkinen@aalto.fi}}



\end{document}